\documentclass{amsart}
\usepackage[all]{xy}
\usepackage{verbatim}
\usepackage{color}
\usepackage{amsthm}
\usepackage{amssymb}
\usepackage[colorlinks=true]{hyperref}

\numberwithin{equation}{section}

\newtheorem{theorem}[equation]{Theorem}
\newtheorem*{theorem*}{Theorem} \newtheorem{lemma}[equation]{Lemma}

\newtheorem*{conjecture*}{Mamma Conjecture}
\newtheorem*{conjecture1*}{Mamma Conjecture (revisited)}
\newtheorem{proposition}[equation]{Proposition}
\newtheorem{corollary}[equation]{Corollary}
\newtheorem*{corollary*}{Corollary}

\theoremstyle{remark}
\newtheorem{definition}[equation]{Definition}

\newtheorem{example}[equation]{Example}

\newtheorem{notation}[equation]{Notation}

\theoremstyle{remark}
\newtheorem{remark}[equation]{Remark}

\newcommand{\cA}{{\mathcal A}}
\newcommand{\cB}{{\mathcal B}}
\newcommand{\cC}{{\mathcal C}}
\newcommand{\cD}{{\mathcal D}}

\newcommand{\cN}{{\mathcal N}}

\newcommand{\cP}{{\mathcal P}}

\newcommand{\cS}{{\mathcal S}}
\newcommand{\cT}{{\mathcal T}}

\newcommand{\bbC}{\mathbb{C}}

\newcommand{\bbF}{\mathbb{F}}
\newcommand{\bbG}{\mathbb{G}}

\newcommand{\bbR}{\mathbb{R}}

\newcommand{\bbQ}{\mathbb{Q}}
\newcommand{\bbZ}{\mathbb{Z}}

\DeclareMathOperator{\NChow}{NChow} 
\DeclareMathOperator{\NNum}{NNum} 

\newcommand{\dgcat}{\mathrm{dgcat}} 

\newcommand{\perf}{\mathrm{perf}}

\newcommand{\dg}{\mathrm{dg}}

\newcommand{\Hom}{\mathrm{Hom}}
\newcommand{\End}{\mathrm{End}}

\newcommand{\rep}{\mathrm{rep}}

\newcommand{\dgHo}{\mathrm{H}^0}

\newcommand{\Hmo}{\mathrm{Hmo}}
\newcommand{\op}{\mathrm{op}}

\newcommand{\too}{\longrightarrow}

\newcommand{\ie}{\textsl{i.e.}\ }

\let\oldmarginpar\marginpar
\def\marginpar#1{\oldmarginpar{\tiny #1}}

\begin{document}

\title[A note on secondary $K$-theory]{A note on secondary $K$-theory}
\author{Gon{\c c}alo~Tabuada}

\address{Gon{\c c}alo Tabuada, Department of Mathematics, MIT, Cambridge, MA 02139, USA}
\email{tabuada@math.mit.edu}
\urladdr{http://math.mit.edu/~tabuada}
\thanks{The author was partially supported by a NSF CAREER Award.}

\date{\today}


\abstract{We prove that To\"en's secondary Grothendieck ring is isomorphic to the Grothendieck ring of smooth proper pretriangulated dg categories previously introduced by Bondal, Larsen and Lunts. Along the way, we show that those short exact sequences of dg categories in which the first term is smooth proper and the second term is proper are necessarily split. As an application, we prove that the canonical map from the derived Brauer group to the secondary Grothendieck ring has the following injectivity properties: in the case of a commutative ring of characteristic zero, it distinguishes between dg Azumaya algebras associated to non-torsion cohomology classes and dg Azumaya algebras associated to torsion cohomology classes (=ordinary Azumaya algebras); in the case of a field of characteristic zero, it is injective; in the case of a field of positive characteristic $p>0$, it restricts to an injective map on the $p$-primary component of the Brauer group.}}

\maketitle
\vskip-\baselineskip
\vskip-\baselineskip



\section{Introduction and statement of results}
A {\em dg category $\cA$}, over a base commutative ring $k$, is a category enriched over complexes of $k$-modules; see \S\ref{sec:dg}. Every (dg) $k$-algebra $A$ gives naturally rise to a dg category with a single object. Another source of examples is provided by schemes since the category of perfect complexes of every quasi-compact quasi-separated $k$-scheme admits a canonical dg enhancement; see Lunts-Orlov \cite{LO}. Following Kontsevich \cite{ENS}, a dg category $\cA$ is called {\em smooth} if it is compact as a bimodule over itself and {\em proper} if the complexes of $k$-modules $\cA(x,y)$ are compact. Examples include the finite dimensional $k$-algebras of finite global dimension (when $k$ is a perfect field) and the dg categories of perfect complexes associated to smooth proper $k$-schemes. Following Bondal-Kapranov \cite{BK}, a dg category $\cA$ is called {\em pretriangulated} if the Yoneda functor $\dgHo(\cA)\to \cD_c(\cA), x \mapsto \widehat{x}$, is an equivalence of categories. As explained in  \S\ref{sec:dg}, every dg category $\cA$ admits a pretriangulated ``envelope'' $\perf_\dg(\cA)$.

Bondal, Larsen, and Lunts introduced in \cite[\S5]{BLL} the {\em Grothendieck ring of smooth proper pretriangulated dg categories $\cP\cT(k)$}. This ring is defined by generators and relations. The generators are the quasi-equivalence classes of smooth proper pretriangulated dg categories\footnote{Bondal, Larsen, and Lunts worked originally with pretriangulated dg categories. In this case the classical Eilenberg's swindle argument implies that the associated Grothendieck ring is trivial. In order to obtain a non-trivial Grothendieck ring, we need to restrict ourselves to smooth proper dg categories; consult \cite[\S7]{Additive} for further details.}. The relations $[\cB]= [\cA] + [\cC] $ arise from the dg categories $\cA, \cC \subseteq \cB$ for which the triangulated subcategories $\dgHo(\cA), \dgHo(\cC) \subseteq \dgHo(\cB)$ are admissible and induce a semi-orthogonal decomposition $\dgHo(\cB)=\langle \dgHo(\cA), \dgHo(\cC)  \rangle$. The multiplication law is given by $\cA \bullet \cB:= \perf_\dg(\cA \otimes^{\bf L} \cB)$, where $-\otimes^{\bf L}-$ stands for the derived tensor product of dg categories. Among other applications, Bondal, Larsen, and Lunts constructed an interesting motivic measure with values in $\cP\cT(k)$. 

More recently, To\"en introduced in \cite[\S5.4]{Toen}\cite{Toen1} a ``categorified'' version of the classical Grothendieck ring named {\em secondary Grothendieck ring $K_0^{(2)}(k)$}. By definition, $K_0^{(2)}(k)$ is the quotient of the free abelian group on the Morita equivalence classes of smooth proper dg categories by the relations $[\cB]= [\cA] + [\cC]$ arising from short exact sequences of dg categories $0 \to \cA \to \cB \to \cC \to 0$. Thanks to the work of Drinfeld \cite[Prop.~1.6.3]{Drinfeld}, the derived tensor product of dg categories endows $K_0^{(2)}(k)$ with a commutative ring structure. Among other applications, the ring $K_0^{(2)}(k)$ was used in the study of derived loop spaces; see \cite{BZ,Toen2,Toen3}.
%
%
\begin{theorem}\label{thm:main}
The rings $\cP \cT(k)$ and $K_0^{(2)}(k)$ are isomorphic. 
\end{theorem}
The proof of Theorem \ref{thm:main} is based on the fact that those short exact sequences of dg categories $0 \to \cA \to \cB \to \cC \to 0$ in which $\cA$ is smooth proper and $\cB$ is proper are necessarily split; see Theorem \ref{thm:main2}. This result is of independent interest. Intuitively speaking, it shows us that the smooth proper dg categories behave as ``injective'' objects. In the setting of triangulated categories, this idea of ``injectivity'' goes back to the pioneering work of Bondal and Kapranov \cite{BK1}.
\section{Applications}
Following To\"en \cite{Azumaya}, a dg $k$-algebra $A$ is called a {\em dg Azumaya algebra} if the underlying complex of $k$-modules is a compact generator of the derived category $\cD(k)$ and the canonical morphism $A^\op \otimes^{\bf L} A \to {\bf R}\Hom(A,A)$ in $\cD(k)$ is invertible. The ordinary Azumaya algebras (see Grothendieck \cite{Grothendieck}) are the dg Azumaya algebras whose underlying complex is $k$-flat and concentrated in degree zero. For every non-torsion {\'e}tale cohomology class $\alpha \in H^2_{\mathrm{et}}(\mathrm{Spec}(k),\bbG_m)$ there exists a dg Azumaya algebra $A_\alpha$, representing $\alpha$, which is {\em not} Morita equivalent to an ordinary Azumaya algebra; see \cite[page~584]{Azumaya}. Unfortunately, the construction of $A_\alpha$ is highly inexplicit; consult \cite[Appendix B]{Separable} for some properties of these mysterious dg algebras. In the case where $k$ is a field, every dg Azumaya algebra is Morita equivalent to an ordinary Azumaya algebra; see \cite[Prop.~2.12]{Azumaya}. 

The {\em derived Brauer group $\mathrm{dBr}(k)$ of $k$} is the set of Morita equivalence classes of dg Azumaya algebras. The (multiplicative) group structure is induced by the derived tensor product of dg categories and the inverse of $A$ is given by $A^\op$. Since every dg Azumaya algebra is smooth proper, we have a canonical map
\begin{equation}\label{eq:canonical}
\mathrm{dBr}(k) \too K_0^{(2)}(k)\,.
\end{equation}
By analogy with the canonical map from the Picard group to the Grothendieck ring
\begin{equation}\label{eq:canonical22}
\mathrm{Pic}(k) \too K_0(k)\,,
\end{equation}
it is natural to ask\footnote{In the case where $k$ is a field, To{\"e}n asked in \cite[\S5.4]{Toen} if the canonical map \eqref{eq:canonical} is non-zero. This follows now automatically from Theorems \ref{thm:injective2} and \ref{thm:injective3}.} if \eqref{eq:canonical} is injective. Note that, in contrast with \eqref{eq:canonical22}, the canonical map \eqref{eq:canonical} does not seem to admit a ``determinant'' map in the converse direction. 
In this note, making use of Theorem \ref{thm:main} and of the recent theory of noncommutative motives (see \S\ref{sec:ncmotives}), we establish several injectivity properties~of~\eqref{eq:canonical}. 

Recall that $k$ has characteristic zero, resp. positive prime characteristic $p>0$, if the kernel of the unique ring homomorphism $\bbZ \to k$ is $\{0\}$, resp. $p\bbZ$.
\begin{theorem}\label{thm:injective1}
Let $k$ be a noetherian commutative ring of characteristic zero, resp. positive prime characteristic $p>0$, and $A$ a dg Azumaya algebra which is not Morita equivalent to an ordinary Azumaya algebra. If  $K_0(k)_\bbQ\simeq \bbQ$, resp. $K_0(k)_{\bbF_p}\simeq \bbF_p$, then the image of $[A]$ under the canonical map \eqref{eq:canonical} is non-trivial. Moreover, when $k$ is of characteristic zero, resp. positive prime characteristic $p>0$, this non-trivial image is different from the images of the ordinary Azumaya algebras, resp. of the ordinary Azumaya algebras whose index is not a multiple of~$p$.
\end{theorem}
As proved by Gabber in \cite[Thm.~II.1]{Gabber}, every torsion {\'e}tale cohomology class $\alpha \in H^2_{\mathrm{et}}(\mathrm{Spec}(k),\bbG_m)_{\mathrm{tor}}$ can be represented by an ordinary Azumaya algebra $A_\alpha$. Therefore, Theorem \ref{thm:injective1} shows us that in some cases the canonical map \eqref{eq:canonical} distinguishes between torsion and non-torsion classes.
\begin{example}
Let $k$ be the noetherian local ring of the singular point of the normal complex algebraic surface constructed by Mumford in \cite[page 16]{Mumford}. As explained by Grothendieck in \cite[page~75]{Grothendieck2}, $k$ is a local $\bbC$-algebra of dimension $2$ whose \'etale cohomology group $H^2_{\mathrm{et}}(\mathrm{Spec}(k),\bbG_m)$ contains non-torsion classes $\alpha$. Therefore, since $K_0(k)\simeq \bbZ$, Theorem \ref{thm:injective1} can be applied to the associated dg Azumaya algebras~$A_\alpha$.
\end{example}
\begin{theorem}\label{thm:injective2}
Let $k$ be a field of characteristic zero. In this case, the canonical map \eqref{eq:canonical} is injective.
\end{theorem}
\begin{example}
\begin{itemize}
\item[(i)] When $k$ is the field of real numbers $\bbR$, we have $\mathrm{Br}(\bbR)\simeq \bbZ/2\bbZ$.
\item[(ii)] When $k$ is the field of $p$-adic numbers $\bbQ_p$, we have $\mathrm{Br}(\bbQ_p)\simeq \bbQ/\bbZ$.
\end{itemize}
\end{example}
\begin{theorem}\label{thm:injective3}
Let $k$ be a field of positive characteristic $p>0$ and $A, B$ two central simple $k$-algebras. If $p\mid \mathrm{ind}(A^\op \otimes B)$, where $\mathrm{ind}$ stands for index,  then the images of $[A]$ and $[B]$ under the canonical map \eqref{eq:canonical} are different. This holds in particular when $\mathrm{ind}(A)$ and $\mathrm{ind}(B)$ are coprime\footnote{When $\mathrm{ind}(A)$ and $\mathrm{ind}(B)$ are coprime we have $\mathrm{ind}(A^\op \otimes B)=\mathrm{ind}(A)\cdot \mathrm{ind}(B)$.} and $p$ divides $\mathrm{ind}(A)$ or $\mathrm{ind}(B)$.
\end{theorem}
\begin{corollary}\label{cor:main}
When $k$ is a field of positive characteristic $p>0$, the restriction of the canonical map \eqref{eq:canonical} to the $p$-primary torsion subgroup $\mathrm{Br}(k)\{p\}$ is injective. Moreover, the image of $\mathrm{Br}(k)\{p\}-0$ is disjoint from the image of $\bigoplus_{q \neq p}\mathrm{Br}(k)\{q\}$.
\end{corollary}
\begin{proof}
The index and the period of a central simple algebra have the same prime factors. Therefore, the proof of the first claim follows from the divisibility relation $\mathrm{ind}(A^\op \otimes B) \mid \mathrm{ind}(A)\cdot \mathrm{ind}(B)$. The proof of the second claim is now clear.
\end{proof}
 \begin{example}
Let $k$ be a field of characteristic $p>0$. Given a character $\chi$ and an invertible element $b \in k^\times$, the associated cyclic algebra $(\chi,b)$ belongs to the $p$-primary torsion subgroup $\mathrm{Br}(k)\{p\}$. Moreover, thanks to the work of Albert (see \cite[Thm.~9.1.8]{Gille}), every element of $\mathrm{Br}(k)\{p\}$ is of this form. Making use of Corollary \ref{cor:main}, we hence conclude that the canonical map \eqref{eq:canonical} distinguishes all these cyclic algebras. Furthermore, the image of $\mathrm{Br}(k)\{p\}-0$ in the secondary Grothendieck ring $K_0^{(2)}(k)$ is disjoint from the image of $\bigoplus_{q \neq p}\mathrm{Br}(k)\{q\}$.

\end{example} 
Every ring homomorphism $k \to k'$ gives rise to the following commutative square
\begin{equation*}
\xymatrix{
\mathrm{dBr}(k) \ar[d]_-{-\otimes^{\bf L}_k k'} \ar[r]^-{\eqref{eq:canonical}} & K_0^{(2)}(k) \ar[d]^-{-\otimes^{\bf L}_k k'} \\
\mathrm{dBr}(k') \ar[r]_-{\eqref{eq:canonical}} & K_0^{(2)}(k')\,.
}
\end{equation*}
By combining it with Theorems \ref{thm:injective2} and \ref{thm:injective3}, we hence obtain the following result:
\begin{corollary}\label{cor:reduction}
Let $A$ and $B$ be dg Azumaya $k$-algebras. If there exists a ring homomorphism $k \to k'$, with $k'$ a field of characteristic zero, resp. positive characteristic $p>0$, such that $[A\otimes^{\bf L}_k k']\neq [B\otimes^{\bf L}_k k']$ in $\mathrm{Br}(k')$, resp. $p\mid \mathrm{ind}((A^\op \otimes^{\bf L}B)\otimes^{\bf L}_k k')$, then the images of $[A]$ and $[B]$ under the canonical map \eqref{eq:canonical} are different.
\end{corollary}
\begin{example}[Local rings]\label{ex:local}
Let $k$ be a complete local ring with residue field $k'$ of characteristic zero, resp. positive characteristic $p>0$. As proved by Auslander-Goldman in \cite[Thm.~6.5]{AG}, the assignment $A \mapsto A \otimes^{\bf L}_k k'$ gives rise to a group isomorphism $\mathrm{Br}(k) \simeq \mathrm{Br}(k')$. Therefore, by combining Corollary \ref{cor:reduction} with Theorem \ref{thm:injective2}, resp. Corollary \ref{cor:main}, we conclude that the restriction of the canonical map \eqref{eq:canonical} to the subgroup $\mathrm{Br}(k)\subset \mathrm{dBr}(k)$, resp. $\mathrm{Br}(k)\{p\}\subset \mathrm{dBr}(k)$, is injective.
\end{example}
\begin{example}[Domains]
Let $k$ be a regular noetherian domain of characteristic zero, resp. positive prime characteristic $p>0$, with field of fractions $k'$. Since $H^1_{\mathrm{et}}(\mathrm{Spec}(k),\bbZ)=0$ and all {\'e}tale cohomology classes of $H^2_{\mathrm{et}}(\mathrm{Spec}(k),\bbG_m)$ are torsion (see \cite[Prop.~1.4]{Grothendieck2}), Gabber's result \cite[Thm.~II.1]{Gabber} implies that the derived Brauer group $\mathrm{dBr}(k)\simeq H^1_{\mathrm{et}}(\mathrm{Spec}(k),\bbZ) \times H^2_{\mathrm{et}}(\mathrm{Spec}(k),\bbG_m)$ agrees with $\mathrm{Br}(k)$. As proved by Auslander-Goldman in \cite[Thm.~7.2]{AG}, the assignment $A \mapsto A \otimes^{\bf L}_k k'$ gives rise to an injective group homomorphism $\mathrm{Br}(k) \to \mathrm{Br}(k')$. Therefore, by combining Corollary \ref{cor:reduction} with Theorem \ref{thm:injective2}, resp. Corollary \ref{cor:main}, we conclude that the canonical map \eqref{eq:canonical}, resp. the restriction of \eqref{eq:canonical} to $\mathrm{Br}(k)\{p\}\subset \mathrm{Br}(k)$, is injective.
\end{example}
\begin{example}[Weyl algebras]
Let $F$ be a field of positive characteristic $p>0$. Thanks to the work of Revoy \cite{Revoy}, the classical Weyl algebra $W_n(F), n \geq 1$, defined as the quotient of $F\langle x_1, \ldots, x_n, \partial_1, \ldots, \partial_n\rangle$ by the relations $[\partial_i, x_j]=\delta_{ij}$, can be considered as an (ordinary) Azumaya algebra over the ring of polynomials $k:=F[x_1^p, \ldots, x_n^p, \partial_1^p, \ldots, \partial_n^p], n\geq 1$. Consider the following composition
$$k:=F[x_1^p, \ldots, x_n^p, \partial_1^p, \ldots, \partial_n^p] \too F[x_1^p, \partial_1^p] \too \mathrm{Frac}(F[x_1^p, \partial_1^p])=:k'\,,$$
where the first homomorphism sends $x^p_i, \partial^p_i, i >1$, to zero and $\mathrm{Frac}(F[x_1^p, \partial_1^p])$ denotes the field of fractions of the integral domain $F[x_1^p, \partial_1^p]$. As explained by Wodzicki in \cite[\S4]{Wodzicki}, we have $\mathrm{ind}(W_n(F)\otimes^{\bf L}_k k')=p$. Therefore, thanks to Corollary \ref{cor:reduction}, we conclude that the image of $W_n(F)$ under \eqref{eq:canonical} is non-trivial.
\end{example}
\begin{example}[Algebras of $p$-symbols]
Let $F$ be a field of positive characteristic $p>0$, $k:=F[x_1^p,\partial_1^p]$ the algebra of polynomials, and $k':=\mathrm{Frac}(F[x_1^p,\partial_1^p])$ the field of fractions. Following Wodzicki \cite[\S1]{Wodzicki}, given elements $a,b \in k$, let us denote by $\cS_{ab}(k) \in {}_p\mathrm{Br}(k)$ the associated (ordinary) Azumaya $k$-algebra of $p$-symbols. For example, when $a=x_1^p$ and $b=\partial_1^p$, we have $\cS_{ab}(k)=W_1(F)$. As proved by Wodzicki in \cite[\S6]{Wodzicki}, we have $\mathrm{ind}(\cS_{ab}(k)\otimes^{\bf L}_k k')=p$ if and only if
\begin{equation}\label{eq:relations}
b\neq c_0^p+ c_1^pa + \cdots + c_{p-1}^p a^{p-1} - c_{p-1} \,\,\mathrm{for}\,\,\mathrm{every}\,\, c_0 + c_1 t + \cdots + c_{p-1} t^{p-1} \in k'[t]\,.
\end{equation}
Therefore, thanks to Corollary \ref{cor:reduction}, we conclude that whenever $a$ and $b$ satisfy condition \eqref{eq:relations} the image of $\cS_{ab}(k)$ under the canonical map \eqref{eq:canonical} is non-trivial.
\end{example}
\begin{remark}[Stronger results]As explained in \S\ref{sec:proof2}, Theorems \ref{thm:injective1}, \ref{thm:injective2} and \ref{thm:injective3}, and Corollaries \ref{cor:main} and \ref{cor:reduction}, follow from stronger analogue results where instead of $K_0^{(2)}(k)$ we consider the Grothendieck ring of the category of noncommutative Chow motives; consult Theorems \ref{thm:injective11}, \ref{thm:injective22} and \ref{thm:injective33}, Corollary \ref{cor:main2}, and Remark~\ref{rk:extension-last}.
\end{remark}
\section{Background on DG categories}\label{sec:dg}
Let $(\cC(k),\otimes, k)$ be the symmetric monoidal category of cochain complexes of $k$-modules. A {\em dg category $\cA$} is a category enriched over $\cC(k)$ and a {\em dg functor} $F\colon\cA\to \cB$ is a functor enriched over $\cC(k)$; consult Keller's ICM survey \cite{ICM-Keller}. Let us denote by $\dgcat(k)$ the category of (small) dg categories and dg functors.

Let $\cA$ be a dg category. The opposite dg category $\cA^\op$ has the same objects and $\cA^\op(x,y):=\cA(y,x)$. The category $\dgHo(\cA)$ has the same objects as $\cA$ and morphisms $\dgHo(\cA)(x,y):=H^0(\cA(x,y))$, where $H^0(-)$ stands for $0^{\mathrm{th}}$-cohomology.

A {\em right dg $\cA$-module} is a dg functor $M\colon\cA^\op \to \cC_\dg(k)$ with values in the dg category $\cC_\dg(k)$ of complexes of $k$-modules. Given $x \in \cA$, let us write $\widehat{x}$ for the Yoneda right dg $\cA$-module defined by $y \mapsto \cA(x,y)$. Let $\cC(\cA)$ be the category of right dg $\cA$-modules. As explained in \cite[\S3.2]{ICM-Keller}, $\cC(\cA)$ carries a Quillen model structure whose weak equivalences, resp. fibrations, are the objectwise quasi-isomorphisms, resp. surjections. The {\em derived category $\cD(\cA)$ of $\cA$} is the associated homotopy category. Let $\cD_c(\cA)$ be the full triangulated subcategory of compact objects. The dg structure of $\cC_\dg(k)$ makes $\cC(\cA)$ naturally into a dg category $\cC_\dg(\cA)$. Let us write $\perf_\dg(\cA)$ for the full dg subcategory of $\cC_\dg(\cA)$ consisting of those cofibrant right dg $\cA$-modules which belong to $\cD_c(\cA)$. Note that we have the Yoneda dg functor $\cA \to \perf_\dg(\cA) \subset \cC_\dg(\cA), x \mapsto \widehat{x}$, and that $\dgHo(\perf_\dg(\cA)) \simeq \cD_c(\cA)$. 

A dg functor $F\colon\cA\to \cB$ is called a {\em quasi-equivalence} if the morphisms of $k$-modules $F(x,y)\colon\cA(x,y) \to \cB(F(x),F(y))$ are quasi-isomorphisms and the induced functor $\dgHo(F)\colon \dgHo(\cA) \to \dgHo(\cB)$ is an equivalence of categories. More generally, $F$ is called a {\em Morita equivalence} if it induces an equivalence of derived categories $\cD(\cA)\to \cD(\cB)$; see \cite[\S4.6]{ICM-Keller}. As proved in \cite[Thm.~5.3]{Additive}, $\dgcat(k)$ carries a Quillen model structure whose weak equivalences are the Morita equivalences. Let us denote by $\Hmo(k)$ the associated homotopy category. 

The tensor product $\cA\otimes\cB$ of two dg categories $\cA$ and $\cB$ is defined as follows: the set of objects is the cartesian product and $(\cA\otimes\cB)((x,w),(y,z)):= \cA(x,y) \otimes \cB(w,z)$. As explained in \cite[\S4.3]{ICM-Keller}, this construction can be derived $-\otimes^{\bf L} -$ giving thus rise to a symmetric monoidal structure on $\Hmo(k)$ with $\otimes$-unit the dg category $k$.

\section{Proof of Theorem \ref{thm:main}}
The smooth proper dg categories can be characterized as the dualizable objects of the symmetric monoidal category $\Hmo(k)$; see \cite[\S5]{CT1}. Consequently, Kontsevich's notions of smoothness and properness are invariant under Morita equivalence. 

Recall from \cite[\S4.6]{ICM-Keller} that a short exact sequence of dg categories is a sequence of morphisms $\cA \to \cB \to \cC$ in the homotopy category $\Hmo(k)$ inducing an exact sequence of triangulated categories $0 \to \cD_c(\cA) \to \cD_c(\cB) \to \cD_c(\cC) \to 0$ in the sense of Verdier. As proved in \cite[Lem.~10.3]{Duke}, the morphism $\cA \to \cB$ is isomorphic to an inclusion of dg categories $\cA \subseteq \cB$ and $\cC$ identifies with Drinfeld's DG quotient $\cB/\cA$. 
\begin{definition}
A short exact sequence of dg categories $0 \to \cA \to \cB \to \cC \to 0$ is called {\em split} if the triangulated subcategory $\cD_c(\cA) \subseteq \cD_c(\cB)$ is admissible.
\end{definition}
\begin{remark}\label{rk:orthogonal}
In the case of a split short exact sequence of dg categories we have an induced equivalence between $\cD_c(\cC)$ and the right orthogonal $\cD_c(\cA)^\perp \subseteq \cD_c(\cB)$. Consequently, we obtain a semi-orthogonal decomposition $\cD_c(\cB) = \langle \cD_c(\cA), \cD_c(\cC)\rangle$. 
\end{remark}
Let us write $K_0^{(2)}(k)^s$ for the ring defined similarly to $K_0^{(2)}(k)$ but with {\em split} short exact sequences of dg categories instead of short exact sequences of dg categories.
\begin{proposition}\label{prop:main1}
The rings $\cP\cT(k)$ and $K_0^{(2)}(k)^s$ are isomorphic.
\end{proposition}
\begin{proof}
The assignment $\cA \mapsto \cA$ clearly sends quasi-equivalence classes of smooth proper pretriangulated dg categories to Morita equivalence classes of smooth proper dg categories. Let $\cA, \cC \subseteq \cB$ be smooth proper pretriangulated dg categories for which the triangulated subcategories $\dgHo(\cA), \dgHo(\cC) \subseteq \dgHo(\cB)$ are admissible and induce a semi-orthogonal decomposition $\dgHo(\cB) = \langle \dgHo(\cA), \dgHo(\cC)\rangle$. Consider the full dg subcategory $\cB'$ of $\cB$ consisting of those objects which belong to $\dgHo(\cA)$ or to $\dgHo(\cC)$. Thanks to the preceding semi-orthogonal decomposition, the inclusion dg functor $\cB'\subseteq \cB$ is a Morita equivalence. Consider also the dg functor $\pi\colon \cB' \to \cC$ which is the identity on $\cC$ and which sends all the remaining objects to a fixed zero object $0$ of $\cC$. Under these notations, we have the following split short exact sequence of dg categories $0 \to \cA \subseteq \cB' \stackrel{\pi}{\to} \cC \to 0$. We hence conclude that the assignment $\cA \mapsto \cA$ gives rise to a group homomorphism $\cP \cT(k) \to K_0^{(2)}(k)^s$. 

As explained in \cite[\S5]{Additive}, the pretriangulated dg categories can be (conceptually) characterized as the fibrant objects of the Quillen model structure on $\dgcat(k)$; see \S\ref{sec:dg}. Moreover, given a dg category $\cA$, the Yoneda dg functor $\cA \to \perf_\dg(\cA), x \mapsto \widehat{x}$, is a fibrant resolution. This implies that $\cP \cT(k) \to K_0^{(2)}(k)^s$ is moreover a surjective ring homomorphism. It remains then only to show its injectivity. Given a split short exact sequence of smooth proper dg categories $0 \to \cA \to \cB \to \cC \to 0$, which we can assume pretriangulated, recall from Remark \ref{rk:orthogonal} that we have an associated semi-orthogonal decomposition $\dgHo(\cB) = \langle \dgHo(\cA), \dgHo(\cC)\rangle$. Let us write $\cC'$ for the full dg subcategory of $\cB$ consisting of those objects which belong to $\dgHo(\cC)$. Note that $\cC'$ is pretriangulated and quasi-equivalent to $\cC$. Note also that since the triangulated subcategory $\dgHo(\cA) \subseteq \dgHo(\cB)$ is admissible, the triangulated subcategory $\dgHo(\cC') \simeq \dgHo(\cA)^\perp \subseteq \dgHo(\cB)$ is also admissible. We hence conclude that the relation $[\cB]=[\cA]+[\cC] \Leftrightarrow [\cB]= [\cA] + [\cC']$ holds in $\cP \cT(k)$, and consequently that the surjective ring homomorphism $\cP\cT(k) \to K_0^{(2)}(k)^s$ is moreover injective.
\end{proof}
Thanks to Proposition \ref{prop:main1}, the proof of Theorem \ref{thm:main} follows now automatically from the following result of independent interest:
\begin{theorem}\label{thm:main2}
Let $0 \to \cA \to \cB \to \cC \to 0$ be a short exact sequence of dg categories. If $\cA$ is smooth proper and $\cB$ is proper, then the sequence is split.
\end{theorem}
\begin{proof}
Without loss of generality, we can assume that the dg categories $\cA$ and $\cB$ are pretriangulated. Let us prove first that the triangulated subcategory $\dgHo(\cA) \subseteq \dgHo(\cB)$ is right admissible, \ie that the inclusion functor admits a right adjoint. Given an object $z \in \cB$, consider the following composition 
\begin{equation}\label{eq:functor}
\dgHo(\cA)^\op \stackrel{\dgHo(\cB(-,z))}{\too} \dgHo(\perf_\dg(k))\simeq \cD_c(k) \stackrel{H^0(-)}{\too}\mathrm{mod}(k)
\end{equation} 
with values in the category of finitely generated $k$-modules. Thanks to Proposition \ref{prop:corepresentability} (with $F=\cB(-,z)$), the functor \eqref{eq:functor} is representable. Let us denote by $x$ the representing object. Since the composition \eqref{eq:functor} is naturally isomorphic to the (contravariant) functor $\Hom_{\dgHo(\cB)}(-,z)\colon \dgHo(\cA)^\op \to \mathrm{mod}(k)$, we have $\Hom_{\dgHo(\cA)}(y,x)\simeq \Hom_{\dgHo(\cB)}(y,z)$ for every $y \in \cA$. By taking $y=x$, we hence obtain a canonical morphism $\eta\colon x \to z$ and consequently a distinguished triangle $x \stackrel{\eta}{\to} z\to \mathrm{cone}(\eta) \to \Sigma(x)$ in the triangulated category $\dgHo(\cB)$. The associated long exact sequences allow us then to conclude that $\mathrm{cone}(\eta)$ belongs to the right orthogonal $\dgHo(\cA)^\perp\subseteq \dgHo(\cB)$. This implies that the triangulated subcategory $\dgHo(\cA) \subseteq \dgHo(\cB)$ is right admissible. The proof of left admissibility is similar: simply replace $\cB(-,z)$ by the covariant dg functor $\cB(z,-)$; see Remark \ref{rk:duality}.
\end{proof}
\begin{notation}[Bimodules]\label{not:bimodules}
Let $\cA$ and $\cB$ be two dg categories. A {\em dg $\cA\text{-}\cB$ bimodule} is a dg functor $\mathrm{B}\colon\cA \otimes^{\bf L} \cB^\op\to \cC_\dg(k)$, \ie a right dg $(\cA^\op \otimes^{\bf L} \cB)$-module. Associated to a dg functor $F\colon \cA \to \cB$, we have the dg $\cA\text{-}\cB$-bimodule 
\begin{eqnarray}\label{eq:bimodule2}
{}_F\mathrm{B}\colon\cA\otimes^{\bf L} \cB^\op \too \cC_\dg(k) && (x,z) \mapsto \cB(z,F(x))\,.
\end{eqnarray}
Let us write $\rep(\cA,\cB)$ for the full triangulated subcategory of $\cD(\cA^\op\otimes^{\bf L} \cB)$ consisting of those dg $\cA\text{-}\cB$ bimodules $\mathrm{B}$ such that for every object $x \in \cA$ the associated right dg $\cB$-module $\mathrm{B}(x,-)$ belongs to $\cD_c(\cB)$. Similarly, let $\rep_\dg(\cA,\cB)$ be the full dg subcategory of $\cC_\dg(\cA^\op \otimes^{\bf L} \cB)$ consisting of those cofibrant right dg $\cA\text{-}\cB$ bimodules which belong to $\rep(\cA,\cB)$. By construction,~$\dgHo(\rep_\dg(\cA,\cB))\simeq \rep(\cA,\cB)$. 
\end{notation}

\begin{proposition}[Representability]\label{prop:corepresentability}
Let $\cA$ be a smooth proper pretriangulated dg category and $G\colon \dgHo(\cA)^\op \to \mathrm{mod}(k)$ be a (contravariant) functor with values in the category of finitely generated $k$-modules. Assume that there exists a dg functor $F\colon\cA^\op \to \perf_\dg(k)$ and a natural isomorphism between $G$ and the composition
$$ \dgHo(\cA)^\op \stackrel{\dgHo(F)}{\too} \dgHo(\perf_\dg(k)) \simeq \cD_c(k) \stackrel{H^0(-)}{\too} \mathrm{mod}(k)\,.$$
Under these assumptions, the functor $G$ is representable.
\end{proposition}
\begin{remark}\label{rk:duality}
Given a smooth proper pretriangulated dg category $\cA$, the opposite dg category $\cA^\op$ is also smooth, proper, and pretriangulated. Therefore, Proposition \ref{prop:corepresentability} (with $\cA$ replaced by $\cA^\op$) is also a corepresentability result.
\end{remark}
\begin{proof}
Following Notation \ref{not:bimodules}, let ${}_F\mathrm{B} \in \rep(\cA^\op,\perf_\dg(k))$ be the dg $\cA^\op\text{-}\perf_\dg(k)$ bimodule associated to the dg functor $F$. Thanks to Lemma \ref{lem:aux} below, there exists an object $x \in \cA$ and an isomorphism in the triangulated category $\rep(\cA^\op,\perf_\dg(k))$ between the dg $\cA^\op\text{-}\perf_\dg(k)$ bimodules ${}_F \mathrm{B}$ and ${}_{\widehat{x}}\mathrm{B}$. Making use of the functor 
\begin{eqnarray}
\rep(\cA^\op,\perf_\dg(k))\too \mathrm{Fun}_{\Delta}(\dgHo(\cA)^\op, \cD_c(k)) && \mathrm{B} \mapsto - \otimes^{\bf L}_{\cA^\op} \mathrm{B}\,,
\end{eqnarray}
where $\mathrm{Fun}_\Delta(-,-)$ stands for the category of triangulated functors, we obtain an isomorphism between the functors $-\otimes^{\bf L}_{\cA^\op} {}_F\mathrm{B}\simeq \dgHo(F)$ and $-\otimes^{\bf L}_{\cA^\op} {}_{\widehat{x}}\mathrm{B}\simeq \dgHo(\widehat{x})$. By composing them with $H^0(-)\colon \cD_c(k) \to \mathrm{mod}(k)$, we hence conclude that $G$ is naturally isomorphic to the representable functor $\Hom_{\dgHo(\cA)}(-,x)$.
%
%
%
%
%
\end{proof}
\begin{lemma}\label{lem:aux}
Given a smooth proper pretriangulated dg category $\cA$, the dg functor 
\begin{eqnarray}\label{eq:quasi}
\cA \too \rep_\dg(\cA^\op,\perf_\dg(k))&& x \mapsto {}_{\widehat{x}}\mathrm{B}
\end{eqnarray} 
is a quasi-equivalence.
\end{lemma}
\begin{proof}
As proved in \cite[Thm.~5.8]{CT1}, the dualizable objects of the symmetric monoidal category $\Hmo(k)$ are the smooth proper dg categories. Moreover, the dual of a smooth proper dg category $\cA$ is the opposite dg category $\cA^\op$ and the evaluation morphism is given by the following dg functor \begin{eqnarray}\label{eq:bimodule}
\cA\otimes^{\bf L} \cA^\op \too \perf_\dg(k) && (x,y) \mapsto \cA(y,x)\,.
\end{eqnarray}
The symmetric monoidal category $\Hmo(k)$ is closed; see \cite[\S4.3]{ICM-Keller}. Given dg categories $\cA$ and $\cB$, their internal Hom is given by $\rep_\dg(\cA,\cB)$. Therefore, by adjunction, \eqref{eq:bimodule} corresponds to the dg functor \eqref{eq:quasi}. Thanks to the unicity of dualizable objects, we hence conclude that \eqref{eq:quasi} is a Morita equivalence. The proof follows now from the fact that a Morita equivalence between pretriangulated dg categories is necessarily a quasi-equivalence; see \cite[\S5]{Additive}.
\end{proof}
\subsection*{Alternative proof of Theorem \ref{thm:main2} when $k$ is a field}
Without loss of generality, we can assume that the dg categories $\cA$ and $\cB$ are pretriangulated. Note first that $\cB$ is proper if and only if we have $\sum_n \mathrm{dim}\, \Hom_{\dgHo(\cB)}(w,z[n])< \infty$ for any two objects $w$ and $z$. Since the dg category $\cA$ is smooth smooth, the triangulated category $\dgHo(\cA)$ admits a strong generator in the sense of Bondal-Van den Bergh; see Lunts \cite[Lem.~3.5 and 3.6]{Lunts}. Using the fact that the (contravariant) functor \eqref{eq:functor} is cohomological and that the triangulated category $\dgHo(\cA)$ is idempotent complete, we hence conclude from Bondal-Van den Bergh's powerful (co)representability result \cite[Thm.~1.3]{BV} that the functor $\Hom_{\dgHo(\cB)}(-,z)\colon \dgHo(\cA)^\op \to \mathrm{vect}(k)$ is representable. The remaining of the proof is now similar. 
\begin{remark}[Orlov's regularity]
Following Orlov \cite[Def.~3.13]{Orlov}, a dg category $\cA$ is called {\em regular} if the triangulated category $\cD_c(\cA)$ admits a strong generator in the sense of Bondal-Van den Bergh. Examples include the dg categories of perfect complexes associated to regular separated noetherian $k$-schemes. Smoothness implies regularity but the converse does not hold. The preceding proof shows us that Theorem \ref{thm:main2} holds more generally with $\cA$ a regular proper dg category.
\end{remark}



\section{Noncommutative motives}\label{sec:ncmotives}
For a survey, resp. book, on noncommutative motives, we invite the reader to consult \cite{Buenos}, resp. \cite{book}. Let $\cA$ and $\cB$ be two dg categories. As proved in \cite[Cor.~5.10]{Additive}, we have an identification between $\Hom_{\Hmo(k)}(\cA,\cB)$ and the isomorphism classes of the category $\rep(\cA,\cB)$, under which the composition law of $\Hmo(k)$ corresponds to the derived tensor product of bimodules. Since the dg $\cA\text{-}\cB$ bimodules \eqref{eq:bimodule2} belong to $\rep(\cA,\cB)$, we hence obtain a symmetric monoidal functor
\begin{eqnarray}\label{eq:func1}
\dgcat(k) \too \Hmo(k) & \cA\mapsto \cA & F \mapsto {}_F \mathrm{B}\,.
\end{eqnarray}
The {\em additivization} of $\Hmo(k)$ is the additive category $\Hmo_0(k)$ with the same objects as $\Hmo(k)$ and with morphisms given by $\Hom_{\Hmo_0(k)}(\cA,\cB):=K_0\rep(\cA,\cB)$, where $K_0\rep(\cA,\cB)$ stands for the Grothendieck group of the triangulated category $\rep(\cA,\cB)$. The composition law is induced by the derived tensor product of bimodules and the symmetric monoidal structure extends by bilinearity from $\Hmo(k)$ to $\Hmo_0(k)$. Note that we have a symmetric monoidal functor
\begin{eqnarray}\label{eq:func2}
\Hmo(k) \too \Hmo_0(k) & \cA \mapsto \cA & \mathrm{B} \mapsto [\mathrm{B}]\,.
\end{eqnarray}
Given a commutative ring of coefficients $R$, the {\em $R$-linearization} of $\Hmo_0(k)$ is the $R$-linear category $\Hmo_0(k)_R$ obtained by tensoring the morphisms of $\Hmo_0(k)$ with $R$. Note that $\Hmo_0(k)_R$ inherits an $R$-linear symmetric monoidal structure and that we have the symmetric monoidal functor
\begin{eqnarray}\label{eq:func3}
\Hmo_0(k) \too \Hmo_0(k)_R & \cA \mapsto \cA & [\mathrm{B}] \mapsto [\mathrm{B}]_R\,.
\end{eqnarray}
Let us denote by $U(-)_R\colon \dgcat(k) \to \Hmo_0(k)_R$ the composition $\eqref{eq:func3}\circ \eqref{eq:func2}\circ \eqref{eq:func1}$.
\subsection*{Noncommutative Chow motives}
The category of {\em noncommutative Chow motives} $\NChow(k)_R$ is defined as the idempotent completion of the full subcategory of $\Hmo_0(k)_R$ consisting of the objects $U(\cA)_R$ with $\cA$ a smooth proper dg category. This category is not only $R$-linear and idempotent complete, but also additive and rigid\footnote{Recall that a symmetric monoidal category is called {\em rigid} if all its objects are dualizable.} symmetric monoidal; see \cite[\S4]{Buenos}. Given dg categories $\cA$ and $\cB$, with $\cA$ smooth proper, we have $\rep(\cA,\cB) \simeq \cD_c(\cA^\op \otimes^{\bf L} \cB)$. Hence, we obtain~isomorphisms
$$ \Hom_{\NChow(k)_R}(U(\cA)_R,U(\cB)_R):=K_0(\rep(\cA,\cB))_R\simeq K_0(\cA^\op \otimes^{\bf L}\cB)_R\,.$$
When $R=\bbZ$, we write $\NChow(k)$ instead of $\NChow(k)_\bbZ$ and $U$ instead of $U(-)_\bbZ$.
\subsection*{Noncommutative numerical motives}
Given an $R$-linear, additive, rigid symmetric monoidal category $\cC$, its {\em $\cN$-ideal} is defined as follows
$$ \cN(a,b):=\{f \in \Hom_\cC(a,b)\,|\, \forall g \in \Hom_\cC(b,a)\,\,\mathrm{we}\,\,\mathrm{have}\,\,\mathrm{tr}(g\circ f)=0 \}\,,$$
where $\mathrm{tr}(g\circ f)$ stands for the categorical trace of the endomorphism $g\circ f$. The category of {\em noncommutative numerical motives} $\NNum(k)_R$ is defined as the idempotent completion of the quotient of $\NChow(k)_R$ by the $\otimes$-ideal $\cN$. By construction, this category  is $R$-linear, additive, rigid symmetric monoidal, and idempotent complete. 
\begin{notation}\label{not:CSA}
In the case where $k$ is a field, we write $\mathrm{CSA}(k)_R$ for the full subcategory of $\NNum(k)_R$ consisting of the objects $U(A)_R$ with $A$ a central simple $k$-algebra, and $\mathrm{CSA}(k)^\oplus_R$ for the closure of $\mathrm{CSA}(k)_R$ under finite direct sums.
\end{notation}
The following result is a slight variant of \cite[Thm.~1.10]{MT}.
\begin{theorem}[Semi-simplicity]\label{thm:semi}
Let $k$ be a commutative ring of characteristic zero, resp. positive prime characteristic $p>0$, and $R$ a field with the same characteristic. If $K_0(k)_\bbQ\simeq \bbQ$, resp. $K_0(k)_{\bbF_p}\simeq \bbF_p$, then $\NNum(k)_R$ is abelian semi-simple.
\end{theorem}
\begin{proof}
As explained in \cite[Example 8.9]{CT1}, Hochschild homology gives rise to an additive symmetric monoidal functor $HH\colon \Hmo_0(k) \to \cD(k)$. The dualizable objects of the derived category $\cD(k)$ are the compact ones. Therefore, since the symmetric monoidal subcategory $\NChow(k)$ of $\Hmo_0(k)$ is rigid and every symmetric monoidal functor preserves dualizable objects, the preceding functor restricts to an additive symmetric monoidal functor
\begin{eqnarray}\label{eq:HH}
\NChow(k) \too \cD_c(k) && U(\cA) \mapsto HH(\cA)\,.
\end{eqnarray}
Assume first that $k$ is of characteristic zero. Let us write $k_\bbQ$ for the localization of $k$ at the multiplicative set $\bbZ\backslash\{0\}$. Choose a maximal ideal $\mathfrak{m}$ of $k_\bbQ$ and consider the associated residue field $F:=k_\bbQ/\mathfrak{m}$. By composing \eqref{eq:HH} with the base-change (derived) symmetric monoidal functors $\cD_c(k) \to \cD_c(k_\bbQ)\to \cD_c(F)$, we hence obtain an induced $\bbQ$-linear symmetric monoidal functor
\begin{eqnarray}\label{eq:HH1}
\NChow(k)_\bbQ \too \cD_c(F) && U(\cA)_\bbQ \mapsto HH(\cA) \otimes^{\bf L}_k F\,.
\end{eqnarray}
Since by assumption we have $\mathrm{End}_{\NChow(k)_\bbQ}(U(k)_\bbQ)=K_0(k)_\bbQ\simeq \bbQ$, we hence conclude from Andr{\'e}-Kahn's general results \cite[Thm.~1 a)]{AK1}\cite[Thm.~A.2.10]{AK}, applied to the functor \eqref{eq:HH1}, that the category $\NNum(k)_\bbQ$ is abelian semi-simple. The proof follows now from the fact that the $\otimes$-ideal $\cN$ is compatible with change of coefficients along the field extension $R/\bbQ$; consult  \cite[Prop.~1.4.1]{Brug} for further details.

Assume now that $k$ is of positive prime characteristic $p>0$. Choose a maximal ideal $\mathfrak{m}$ of $k$ and consider the associated residue field $F:=k/\mathfrak{m}$. As in the characteristic zero case, we obtain an induced $\bbF_p$-linear symmetric monoidal functor
\begin{eqnarray*}
\NChow(k)_{\bbF_p} \too \cD_c(F) && U(\cA)_{\bbF_p} \mapsto HH(\cA) \otimes^{\bf L}_k F\,,
\end{eqnarray*}
which allows us to conclude that the category $\NNum(k)_R$ is abelian semi-simple. 
\end{proof}
\section{Proof of Theorems \ref{thm:injective1}, \ref{thm:injective2}, and \ref{thm:injective3}}\label{sec:proof2}
We start by studying the noncommutative Chow motives of dg Azumaya algebras. These results are of independent interest.
\begin{proposition}\label{prop:dgAzumaya}
Let $k$ be a commutative ring and $A$ a dg Azumaya $k$-algebra which is not Morita equivalent to an ordinary Azumaya algebra. If $k$ is noetherian, then we have $U(A)_\bbQ\not\simeq U(k)_\bbQ$ in $\NChow(k)_\bbQ$ and $U(A)_{\bbF_q} \not\simeq U(k)_{\bbF_q}$ in $\NChow(k)_{\bbF_q}$ for every prime number $q$. 
\end{proposition}
\begin{proof}
As proved in \cite[Thm.~B.15]{Separable}, we have $U(A)_\bbQ\not\simeq U(k)_\bbQ$ in $\NChow(k)_\bbQ$. The proof that $U(A)_{\bbF_q} \not\simeq U(k)_{\bbF_q}$ in $\NChow(k)_{\bbF_q}$ is similar: simply further assume that $q$ does not divides the positive integers $m,n>0$ used in {\em loc. cit.}
\end{proof}
\begin{proposition}\label{prop:index}
Let $k$ be a field, $A$ and $B$ two central simple $k$-algebras, and $R$ a commutative ring of positive prime characteristic $p>0$.
\begin{itemize}
\item[(i)] If $p\mid\mathrm{ind}(A^\op \otimes B)$, then $U(A)_R \not\simeq U(B)_R$ in $\NChow(k)_R$. Moreover, we have $\Hom_{\NNum(k)_R}(U(A)_R, U(B)_R)=\Hom_{\NNum(k)_R}(U(B)_R, U(A)_R)=0$.
\item[(ii)] If $p\nmid\mathrm{ind}(A^\op \otimes B)$ and $R$ is a field, then $U(A)_R\simeq U(B)_R$ in $\NChow(k)_R$.
\end{itemize}
\end{proposition}
\begin{proof}
As explained in the proof of \cite[Prop.~2.25]{Separable}, we have natural identifications $ \Hom_{\NChow(k)}(U(A),U(B))\simeq \bbZ$, under which the composition law (in $\NChow(k)$)
$$ \Hom(U(A),U(B)) \times \Hom(U(B),U(C)) \too \Hom(U(A),U(C))$$
corresponds to the bilinear pairing
\begin{eqnarray*}
\bbZ \times \bbZ \too \bbZ && (n,m) \mapsto n \cdot \mathrm{ind}(A^\op \otimes B) \cdot \mathrm{ind}(B^\op \otimes C) \cdot m\,.
\end{eqnarray*}
Hence, we obtain natural identifications $\Hom_{\NChow(k)_R}(U(A)_R,U(B)_R)\simeq R$. Moreover, since $\mathrm{ind}(A^\op \otimes B) = \mathrm{ind}(B^\op \otimes A)$, the composition law (in $\NChow(k)_R$)
$$ \Hom(U(A)_R, U(B)_R) \times \Hom(U(B)_R, U(A)_R) \too \Hom(U(A)_R, U(A)_R)$$
corresponds to the bilinear pairing
\begin{eqnarray}\label{eq:pairing-last}
R \times R \too R && (n,m) \mapsto n \cdot \mathrm{ind}(A^\op \otimes B)^2 \cdot m\,;
\end{eqnarray}
similarly with $A$ and $B$ replaced by $B$ and $A$, respectively. 

If $p \mid \mathrm{ind}(A^\op \otimes B)$, then the bilinear pairing \eqref{eq:pairing-last} is zero. This implies that $U(A)_R \not\simeq U(B)_R$ in $\NChow(k)_R$. Moreover, since the categorical trace of the zero endomorphism is zero, we conclude that all the elements of the $R$-modules $\Hom_{\NChow(k)_R}(U(A)_R, U(B)_R)$ and $\Hom_{\NChow(k)_R}(U(B)_R, U(A)_R)$ belong to the $\cN$-ideal. In other words, we have $\Hom_{\NNum(k)_R}(U(A)_R,U(B)_R)=0$ and also $\Hom_{\NNum(k)_R}(U(B)_R,U(A)_R)=0$. This proves item (i). 

If $p \nmid \mathrm{ind}(A^\op \otimes B)$ and $R$ is a field, then $\mathrm{ind}(A^\op \otimes B)$ is invertible in $R$. It follows then from the bilinear pairing \eqref{eq:pairing-last} that $U(A)_R\simeq U(B)_R$ in $\NChow(k)_R$. This proves item (ii).
\end{proof}

Let $\cA,\cC \subseteq \cB$ be smooth proper pretriangulated dg categories for which the triangulated subcategories $\dgHo(\cA),\dgHo(\cC)\subseteq \dgHo(\cB)$ are admissible and induce a semi-orthogonal decomposition $\dgHo(\cB)=\langle \dgHo(\cA), \dgHo(\cC)\rangle$. As proved in \cite[Thm.~6.3]{Additive}, the inclusion dg functors $\cA, \cC\subseteq \cB$ induce an isomorphism $U(\cA) \oplus U(\cC) \simeq U(\cB)$ in the additive category $\NChow(k)$. Consequently, if we denote by $K_0(\NChow(k))$ the Grothendieck ring of the symmetric monoidal additive $\NChow(k)$, we obtain a well-defined ring homomorphism
\begin{eqnarray*}
\cP\cT(k) \too K_0(\NChow(k)) && [\cA] \mapsto [U(\cA)]\,.
\end{eqnarray*}
By precomposing it with the isomorphism $K_0^{(2)}(k)\simeq \cP\cT(k)$ of Theorem \ref{thm:main} and with  \eqref{eq:canonical}, we hence obtain the canonical map
\begin{eqnarray}\label{eq:canonical2}
\mathrm{dBr}(k) \too K_0(\NChow(k)) && [A] \mapsto [U(A)]\,.
\end{eqnarray}
The proof of Theorem \ref{thm:injective1} follows now from the following result:
\begin{theorem}\label{thm:injective11}
Let $k$ be a noetherian commutative ring of characteristic zero, resp. positive prime characteristic $p>0$, and $A$ a dg Azumaya algebra which is not Morita equivalent to an ordinary Azumaya algebra. If  $K_0(k)_\bbQ\simeq \bbQ$, resp. $K_0(k)_{\bbF_p}\simeq \bbF_p$, then the image of $[A]$ under the canonical map \eqref{eq:canonical2} is non-trivial. Moreover, when $k$ is of characteristic zero, resp. positive prime characteristic $p>0$, this non-trivial image is different from the images of the ordinary Azumaya algebras, resp. of the ordinary Azumaya algebras whose index is not a multiple~of~$p$.
\end{theorem}
\begin{proof}
Similarly to ordinary Azumaya algebras (see \cite[Lem.~8.10]{JIMJ}), we have the following equivalence of symmetric monoidal triangulated categories 
\begin{eqnarray*}
\cD_c(k) \stackrel{\simeq}{\too} \cD_c(A^\op \otimes^{\bf L} A) && M \mapsto M\otimes^{\bf L} A\,,
\end{eqnarray*}
where the symmetric monoidal structure on $\cD_c(k)$, resp. $\cD_c(A^\op \otimes^{\bf L} A)$, is induced by $-\otimes^{\bf L}-$, resp. $-\otimes^{\bf L}_{A}-$. Consequently, we obtain an induced ring isomorphism
\begin{equation}\label{eq:induced1}
\End_{\NChow(k)}(U(k)) \stackrel{\sim}{\too} \End_{\NChow(k)}(U(A))\,.
\end{equation}
Let us prove the first claim. Assume that $k$ is of characteristic zero; the proof of the cases where $k$ is of positive prime characteristic $p>0$ is similar. By definition of the category of noncommutative Chow motives, the left-hand side of \eqref{eq:induced1} is given by the Grothendieck ring $K_0(k)$. Therefore, the assumption $K_0(k)_\bbQ\simeq~\bbQ$ combined with the isomorphism \eqref{eq:induced1} implies that $\End_{\NChow(k)_\bbQ}(U(k)_\bbQ)\simeq \bbQ$ and $\End_{\NChow(k)_\bbQ}(U(A)_\bbQ)\simeq \bbQ$. By construction of the category of noncommutative numerical motives, we have $\End_{\NNum(k)_\bbQ}(U(k)_\bbQ)\simeq \bbQ$. Using the fact that $U(A)_\bbQ \in \NChow(k)_\bbQ$ is a $\otimes$-invertible object and that the $\bbQ$-linear quotient functor $\NChow(k)_\bbQ \to \NNum(k)_\bbQ$ is symmetric monoidal, we hence conclude that $\End_{\NNum(k)_\bbQ}(U(A)_\bbQ)$ is also isomorphic to $\bbQ$. This gives rise to the implication:
\begin{equation}\label{eq:implication}
U(A)_\bbQ\not\simeq U(k)_\bbQ\,\,\mathrm{in}\,\,\NChow(k)_\bbQ \Rightarrow U(A)_\bbQ\not\simeq U(k)_\bbQ\,\,\mathrm{in}\,\,\NNum(k)_\bbQ\,.
\end{equation}
Note that since the quotient functor $\NChow(k)_\bbQ \to \NNum(k)_\bbQ$ is full, implication \eqref{eq:implication} is equivalent to the fact that every morphism $U(A)_\bbQ \to U(k)_\bbQ$ in $\NChow(k)_\bbQ$ which becomes invertible in $\NNum(k)_\bbQ$ is already invertible in $\NChow(k)_\bbQ$.

Recall from Proposition \ref{prop:dgAzumaya} that $U(A)_\bbQ\not\simeq U(k)_\bbQ$ in $\NChow(k)_\bbQ$ when $k$ is noetherian. Making use of \eqref{eq:implication}, we hence conclude that $U(A)_\bbQ\not\simeq U(k)_\bbQ$ in $\NNum(k)_\bbQ$.
%
By definition, we have $[U(A)]=[U(k)]$ in the Grothendieck ring $K_0(\NChow(k))$ if and only if the following condition holds: 
\begin{equation}\label{eq:iso-key}
\exists\, N\!\!M \in \NChow(k)\,\,\mathrm{such}\,\,\mathrm{that}\,\, U(A) \oplus N\!\!M \simeq U(k) \oplus N\!\!M\,.
\end{equation}
Thanks to Theorem \ref{thm:semi}, the category $\NNum(k)_\bbQ$ is abelian semi-simple. Consequently, it satisfies the cancellation property with respect to direct sums. Therefore, if condition \eqref{eq:iso-key} holds, one would conclude that $U(A)_\bbQ\simeq U(k)_\bbQ$ in $\NNum(k)_\bbQ$, which is a contradiction. This finishes the proof of the first claim. 

Let us now prove the second claim. Note first that by combining Proposition \ref{prop:dgAzumaya} with implication \eqref{eq:implication}, we conclude that
\begin{eqnarray}\label{eq:contra}
& U(A)_\bbQ \not \simeq U(k)_\bbQ\,\,\mathrm{in}\,\,\mathrm{NNum}(k)_\bbQ & \mathrm{resp.}\,\,U(A)_{\bbF_p}\not\simeq U(k)_{\bbF_p} \,\,\mathrm{in}\,\,\mathrm{NNum}(k)_{\bbF_p}\,.
\end{eqnarray}
Let $B$ be an ordinary Azumaya $k$-algebra, resp. an ordinary Azumaya $k$-algebra whose index is not a multiple of $p$. In the latter case, by definition of index, we can assume without loss of generality that the rank of $B$ is not a multiple of $p$. We have $[U(A)]=[U(B)]$ in the Grothendieck ring $K_0(\NChow(k))$ if and only if
\begin{equation}\label{eq:cond-A-B}
\exists\, N\!\!M \in \NChow(k)\,\,\mathrm{such}\,\,\mathrm{that}\,\, U(A) \oplus N\!\!M \simeq U(B) \oplus N\!\!M\,.
\end{equation}
Thanks to Theorem \ref{thm:semi}, the category $\NNum(k)_\bbQ$, resp. $\NNum(k)_{\bbF_p}$, is abelian semi-simple. Consequently, it satisfies the cancellation property with respect to direct sums. Therefore, if condition \eqref{eq:cond-A-B} holds, one would conclude that $U(A)_\bbQ\simeq U(B)_\bbQ$ in $\NNum(k)_\bbQ$, resp. $U(A)_{\bbF_p}\simeq U(B)_{\bbF_p}$ in $\NNum(k)_{\bbF_p}$. On the one hand, \cite[Cor.~B.14]{Separable} implies that $U(B)_\bbQ \simeq U(k)_\bbQ$ in $\NNum(k)_\bbQ$. This contradicts the left-hand side of \eqref{eq:contra}. On the other hand, since the rank of $B$ is invertible in $\bbF_p$, \cite[Cor.~B.14]{Separable} implies that $U(B)_{\bbF_p}\simeq U(k)_{\bbF_p}$. This contradicts the right-hand side of \eqref{eq:contra}. The proof of the second claim is then finished.
\end{proof}
\begin{proposition}\label{prop:Br-graded}
Let $k$ be a field and $R$ a field of positive characteristic $p>0$. In this case, the category $\mathrm{CSA}(k)^\oplus_R$ (see Notation \ref{not:CSA}) is equivalent to the category of $\mathrm{Br}(k)\{p\}$-graded finite dimensional $R$-vector spaces.
\end{proposition}
\begin{proof}
Let $A$ be a central simple $k$-algebra. Similarly to the proof of Theorem \ref{thm:injective11}, we have a ring  isomorphism $\mathrm{End}_{\mathrm{CSA}(k)_R}(U(A)_R)\simeq R$.

Let $A$ and $B$ be central simple $k$-algebras such that $[A], [B] \in \mathrm{Br}(k)\{p\}$ and $[A]\neq [B]$. Since $\mathrm{ind}(A^\op \otimes B) \mid \mathrm{ind}(A^\op) \cdot \mathrm{ind}(B)$ and $[A]\neq [B]$, we have $p \mid \mathrm{ind}(A^\op \otimes~B)$. Therefore, Proposition \ref{prop:index}(i) implies that $\Hom_{\mathrm{CSA}(k)_R}(U(A)_R,U(B)_R) = 0$ and also that $\Hom_{\mathrm{CSA}(k)_R}(U(B)_R,U(A)_R)= 0$.

Let $A$ be a central simple $k$-algebra such that $[A] \in \bigoplus_{q\neq p} \mathrm{Br}(k)\{q\}$. Then, Proposition \ref{prop:index}(ii) implies that $U(A)_R \simeq U(k)_R$ in $\mathrm{CSA}(k)_R$.

The proof follows now automatically from the combination of the above facts.
\end{proof}
The proof of Theorem \ref{thm:injective2} follows now from the following result:
\begin{theorem}\label{thm:injective22}
Let $k$ be a field of characteristic zero. In this case, the canonical map \eqref{eq:canonical2} is injective.
\end{theorem}
\begin{proof}
Let $A$ and $B$ be two central simple $k$-algebras such that $[A]\neq [B]$ in $\mathrm{Br}(k)$. Recall that $\mathrm{ind}(A^\op \otimes B) =1$ if and only if $[A]=[B]$. Therefore, let us choose a prime number $p$ such that $p\mid \mathrm{ind}(A^\op \otimes B)$. Thanks to Proposition \ref{prop:index}(i), we have $U(A)_{\bbF_p}\not\simeq U(B)_{\bbF_p}$ in $\NChow(k)_{\bbF_p}$. Consequently, similarly to implication \eqref{eq:implication}, we have $U(A)_{\bbF_p}\not\simeq U(B)_{\bbF_p}$ in $\NNum(k)_{\bbF_p}$. By definition, we have $[U(A)]=[U(B)]$ in the Grothendieck ring $K_0(\NChow(k))$ if and only if the condition holds:
\begin{equation}\label{eq:cond-last}
\exists\, N\!\!M \in \NChow(k)\,\,\mathrm{such}\,\,\mathrm{that}\,\, U(A) \oplus N\!\!M \simeq U(B) \oplus N\!\!M\,.
\end{equation}
Thanks to Lemma \ref{lem:key} below, if condition \eqref{eq:cond-last} holds, then there exist non-negative integers $n, m \geq 0$ and a noncommutative numerical motive $N\!\!M'$~such~that 
\begin{equation}\label{eq:iso-global}
\oplus^{n+1}_{i=1} U(A)_{\bbF_p} \oplus \oplus^m_{j=1} U(B)_{\bbF_p} \oplus N\!\!M' \simeq \oplus^n_{i=1} U(A)_{\bbF_p} \oplus \oplus^{m+1}_{j=1} U(B)_{\bbF_p} \oplus N\!\!M'
\end{equation}
in $\NNum(k)_{\bbF_p}$. Note that the composition bilinear pairing (in $\NNum(k)_{\bbF_p}$)
\begin{equation}\label{eq:pairing-1}
\Hom(U(A)_{\bbF_p}, N\!\!M') \times \Hom(N\!\!M', U(A)_{\bbF_p}) \too \Hom(U(A)_{\bbF_p}, U(A)_{\bbF_p})
\end{equation}
is zero; similarly for $U(B)_{\bbF_p}$. This follows from the fact that the right-hand side of \eqref{eq:pairing-1} identifies with $\bbF_p$, from the fact that the category $\NNum(k)_{\bbF_p}$ is $\bbF_p$-linear, and from the fact that the noncommutative numerical motive $N\!\!M'$ does not contains $U(A)_{\bbF_p}$ as a direct summand. The composition bilinear pairing (in $\NNum(k)_{\bbF_p}$)
\begin{equation}\label{eq:pairing-2}
\Hom(U(A)_{\bbF_p}, N\!\!M')\times \Hom(N\!\!M', U(B)_{\bbF_p})\too \Hom(U(A)_{\bbF_p}, U(B)_{\bbF_p})
\end{equation}
is also zero; similarly with $A$ and $B$ replaced by $B$ and $A$, respectively. This follows automatically from the fact that the right-hand side of \eqref{eq:pairing-2} is zero; see Proposition \ref{prop:index}(i). Now, note that the triviality of the pairings \eqref{eq:pairing-1}-\eqref{eq:pairing-2} implies that the isomorphism \eqref{eq:iso-global} restricts to an isomorphism
$$
U(A)_{\bbF_p}\oplus \oplus^n_{i=1} U(A)_{\bbF_p} \oplus \oplus^m_{j=1} U(B)_{\bbF_p} \simeq U(B)_{\bbF_p}\oplus \oplus^n_{i=1} U(A)_{\bbF_p} \oplus \oplus^m_{j=1} U(B)_{\bbF_p}
$$
in the category $\mathrm{CSA}(k)_{\bbF_p}^\oplus \subset \NNum(k)_{\bbF_p}$. Since $\mathrm{CSA}(k)_{\bbF_p}^\oplus$ is equivalent to the category of $\mathrm{Br}(k)\{p\}$-graded finite dimensional $\bbF_p$-vector spaces (see Proposition \ref{prop:Br-graded}), it satisfies the cancellation property with respect to direct sums. Consequently, we conclude from the preceding isomorphism that $U(A)_{\bbF_p}\simeq U(B)_{\bbF_p}$ in $\NNum(k)_{\bbF_p}$, which is a contradiction. This finishes the proof.
\end{proof}
\begin{lemma}\label{lem:key}
There exist non-negative integers $n,m \geq 0$ and a noncommutative numerical motive $N\!\!M' \in \NNum(k)_{\bbF_p}$ such that:
\begin{itemize}
\item[(i)] We have $N\!\!M_{\bbF_p}\simeq \oplus^n_{i=1} U(A)_{\bbF_p} \oplus \oplus _{j=1}^m U(B)_{\bbF_p} \oplus N\!\!M'$ in $\NNum(k)_{\bbF_p}$.
\item[(ii)] The noncommutative numerical motive $N\!\!M'$ does not contains $U(A)_{\bbF_p}$ neither $U(B)_{\bbF_p}$ as direct summands.
\end{itemize}
\end{lemma}
\begin{proof}
Recall that the category $\NNum(k)_{\bbF_p}$ is idempotent complete. Therefore, by inductively splitting the (possible) direct summands $U(A)_{\bbF_p}$ and $U(B)_{\bbF_p}$ of the noncommutative numerical motive $N\!\!M_{\bbF_p}$, we obtain an isomorphism
$$ N\!\!M_{\bbF_p} \simeq U(A)_{\bbF_p} \oplus \cdots \oplus U(A)_{\bbF_p} \oplus U(B)_{\bbF_p} \oplus \cdots \oplus U(B)_{\bbF_p} \oplus N\!\!M'$$
in $\NNum(k)_{\bbF_p}$, with $N\!\!M'$ satisfying condition (ii). We claim that the number of copies of $U(A)_{\bbF_p}$ and $U(B)_{\bbF_p}$ is finite; note that this concludes the proof. We will focus ourselves in the case $U(A)_{\bbF_p}$; the proof of the case $U(B)_{\bbF_p}$ is similar. Suppose that the number of copies of $U(A)_{\bbF_p}$ is infinite. Since we have natural isomorphisms
\begin{equation}\label{eq:endos}
\Hom_{\NNum(k)_{\bbF_p}}(U(A)_{\bbF_p}, U(A)_{\bbF_p})\simeq \bbF_p\,,
\end{equation}
this would allows us to construct an infinite sequence $f_1, f_2, \ldots$ of vectors in the $\bbF_p$-vector space $\Hom_{\NNum(k)_{\bbF_p}}(U(A)_{\bbF_p}, N\!\!M_{\bbF_p})$, with $f_i$ corresponding to the element $1 \in \bbF_p$ of \eqref{eq:endos}, such that $f_1, \ldots, f_r$ is linearly independent for every positive integer $r$. In other words, this would allows us to conclude that the $\bbF_p$-vector space $\Hom_{\NNum(k)_{\bbF_p}}(U(A)_{\bbF_p}, N\!\!M_{\bbF_p})$ is infinite dimensional. Recall from the proof of \cite[Prop.~1.4.1]{Brug} that the map $\bbZ \to \bbF_p$ gives rise to a surjective homomorphism 
\begin{equation}\label{eq:surjective}
\Hom_{\NNum(k)}(U(A), N\!\!M)\otimes_\bbZ \bbF_p \twoheadrightarrow \Hom_{\NNum(k)_{\bbF_p}}(U(A)_{\bbF_p}, N\!\!M_{\bbF_p})\,.
\end{equation}
Since, by assumption, the base field  $k$ is of characteristic zero, the abelian group $\Hom_{\NNum(k)}(U(A),N\!\!M)$ is finitely generated; see \cite[Thm.~1.2]{Separable}. Therefore, we conclude that the right-hand side of \eqref{eq:surjective} is a finite dimensional $\bbF_p$-vector space, which is a contradiction. This finishes the proof.
\end{proof}
The proof of Theorem \ref{thm:injective3} follows now from the following result:
\begin{theorem}\label{thm:injective33}
Let $k$ be a field of positive characteristic $p>0$ and $A, B$ two central simple $k$-algebras. If $p\mid\mathrm{ind}(A^\op \otimes B)$, then the images of $[A]$ and $[B]$ under the canonical map \eqref{eq:canonical2} are different. This holds in particular when $\mathrm{ind}(A)$ and $\mathrm{ind}(B)$ are coprime and $p$ divides $\mathrm{ind}(A)$ or $\mathrm{ind}(B)$.
\end{theorem}
\begin{proof}
If $p\mid \mathrm{ind}(A^\op \otimes B)$, then Proposition \ref{prop:index}(i) implies that $U(A)_{\bbF_p}\not\simeq U(B)_{\bbF_p}$ in $\NChow(k)_{\bbF_p}$. Consequently, similarly to implication \eqref{eq:implication}, we conclude that $U(A)_{\bbF_p}\not\simeq U(B)_{\bbF_p}$ in $\NNum(k)_{\bbF_p}$. By definition, we have $[U(A)]=[U(B)]$ in the Grothendieck ring $K_0(\NChow(k))$ if and only if the following condition holds: 
\begin{equation}\label{eq:iso-key1}
\exists\, N\!\!M \in \NChow(k)\,\,\mathrm{such}\,\,\mathrm{that}\,\, U(A) \oplus N\!\!M \simeq U(B) \oplus N\!\!M\,.
\end{equation}
Thanks to Theorem \ref{thm:semi}, the category $\NNum(k)_{\bbF_p}$ is abelian semi-simple. Consequently, it satisfies the cancellation property with respect to direct sums. Therefore, if condition \eqref{eq:iso-key1} holds, one would conclude that $U(A)_{\bbF_p}\simeq U(B)_{\bbF_p}$ in $\NNum(k)_{\bbF_p}$, which is a contradiction. This finishes the proof. 
\end{proof}
\begin{corollary}\label{cor:main2}
When $k$ is a field of positive characteristic $p>0$, the restriction of the canonical map \eqref{eq:canonical2} to the $p$-primary torsion subgroup $\mathrm{Br}(k)\{p\}$ is injective. Moreover, the image of $\mathrm{Br}(k)\{p\}-0$ is disjoint from the image of $\bigoplus_{q \neq p}\mathrm{Br}(k)\{q\}$.
\end{corollary}
\begin{remark}\label{rk:extension-last}
As proved in \cite[Thm.~7.1]{NCArtin}, every ring homomorphism $k \to k'$ gives rise to the following commutative square
\begin{equation*}
\xymatrix{
\mathrm{dBr}(k) \ar[d]_-{-\otimes^{\bf L}_k k'} \ar[r]^-{\eqref{eq:canonical2}} & K_0(\NChow(k))\ar[d]^-{-\otimes^{\bf L}_k k'} \\
\mathrm{dBr}(k') \ar[r]_-{\eqref{eq:canonical2}} & K_0(\NChow(k'))\,.
}
\end{equation*}
Therefore, by combining it with Theorems \ref{thm:injective22} and \ref{thm:injective33}, we conclude that Corollary \ref{cor:reduction} also holds with \eqref{eq:canonical} replaced by \eqref{eq:canonical2}.
\end{remark}
\medbreak\noindent\textbf{Acknowledgments:}
The author is very grateful to Ofer Gabber, Maxim Kontsevich, and Michel Van den Bergh for useful discussions. He also would like to thank the Institut des Hautes {\'E}tudes Scientifiques (IH{\'E}S) for its hospitality, excellent working conditions, and stimulating atmosphere. 

\end{document}

\end{proof}